\documentclass[12pt]{amsart}
\usepackage{graphicx}
\usepackage{amsfonts}
\usepackage{amssymb}
\usepackage{latexsym}
\usepackage{amscd}

\newtheorem{thm}{Theorem}[section]

\newtheorem{lem}[thm]{Lemma}
\newtheorem{prop}[thm]{Proposition}

\newenvironment{pf}{\proof[\proofname]}{\endproof}
\newenvironment{pf*}[1]{\proof[#1]}{\endproof}
\usepackage{euscript}

\usepackage[OT2,OT1]{fontenc}

\newcommand{\fxpt}{{{{\mathbf v}_*}}}
\newcommand{\bfv}{{{\mathbf v}}}

\newcommand{\cal}[1]{{\mathcal #1}}

\theoremstyle{definition}
\newtheorem{defn}{Definition}[section]

\theoremstyle{remark}

\newcommand{\diam}{\operatorname{diam}}
\newcommand{\dist}{\operatorname{dist}}

\newcommand{\cl}{\operatorname{cl}}
\renewcommand{\mod}{\operatorname{mod}}
\newcommand{\tl}{\tilde}
\newcommand{\wtl}{\widetilde}
\newcommand{\eps}{\epsilon}

\newcommand{\aaa}[1]{{{\mathbf{#1}}}}

\newcommand{\cu}{{{\mathbf C}_{\mathcal U}}}

\newcommand{\curr}{{{\aaa C}_\cU^\RR}}

\newcommand{\hol}{{\aaa H}}
\newcommand{\mfld}{{\curr}}

\numberwithin{equation}{section}
\newcommand{\thmref}[1]{Theorem~\ref{#1}}
\newcommand{\propref}[1]{Proposition~\ref{#1}}

\newcommand{\lemref}[1]{Lemma~\ref{#1}}

\newcommand{\ang}[2]{\widehat{(#1,#2)}}
\newcommand{\C}[1]{{\Bbb C_{#1}}}

\newcommand{\cA}{{\cal A}}

\newcommand{\cU}{{\mathcal U}}

\newcommand{\cG}{{\cal G}}

\newcommand{\cP}{{\cal P}}
\newcommand{\cC}{{\cal C}}
\newcommand{\cH}{{\cal H}}
\newcommand{\cR}{{\cal R}}
\newcommand{\cL}{{\cal L}}

\newcommand{\cE}{{\cal E}}

\newcommand{\cK}{{\cal K}}

\newcommand{\CC}{{\Bbb C}}

\newcommand{\RR}{{\Bbb R}}
\newcommand{\TT}{{\Bbb T}}
\newcommand{\ZZ}{{\Bbb Z}}
\newcommand{\NN}{{\Bbb N}}
\newcommand{\DD}{{\Bbb D}}
\newcommand{\HH}{{\Bbb H}}
\newcommand{\QQ}{{\Bbb Q}}

\newcommand{\cren}{\cR_{\text cyl}}

\newcommand{\sm}{\setminus}

\begin{document}
\addtolength{\evensidemargin}{-0.7in}
\addtolength{\oddsidemargin}{-0.7in}

\title[Bi-cubic circle maps]
{Renormalization of  bi-cubic circle maps.}
\thanks{This work was partially supported by NSERC Discovery Grant}
\author{Michael Yampolsky}
\begin{abstract}We develop a renormalization theory for analytic homeomorphisms of the circle with two cubic critical points. We prove a renormalization hyperbolicity theorem. As a basis for the proofs, we develop complex {\it a priori} bounds for multi-critical circle maps.
\end{abstract}
\date{\today}
\maketitle

\section{Introduction}
In this paper we develop a renormalization theory for analytic critical circle maps with several critical points. The principal result is a renormalization hyperbolicity theorem for {\it bi-cubic} circle maps, that is, analytic homeomorphisms of the circle with two critical points, both of which are of cubic type. We construct periodic orbits for renormalization for such maps (Theorem~\ref{main1}), and show (Theorem~\ref{cor-hyperb}) that these orbits are hyperbolic, with {\it two} unstable directions.

The proof of Theorem~\ref{main1} uses a new version of our complex {\it a priori} bounds \cite{Ya1}. The proof of the bounds is of an independent importance, and we give it for {\it arbitrary} multicritical circle maps, not just bi-cubic ones (Theorem~\ref{thm:bounds}).

\section{Preliminaries}

\subsection{Multicritical circle maps and commuting pairs}
A multicritical circle map $f:\TT\to\TT$  is a $C^3$-smooth homeomorphism of the circle $\TT$ which has finitely many critical points
$c_0,\ldots,c_l$ so that:
\begin{equation}
  \label{eq:order}
  f(x)-f(c_j)=a_j(x-c_j)^{k_j}\phi_j(x),\text{ for }x\approx c_j,
  \end{equation}
where $k_j\in 2\ZZ+1$ and $\phi_j$ is a local $C^3$-diffeomorphism. We will say that a critical point described by the equation (\ref{eq:order}) has order $k_j$.

The same definition becomes much simpler in the analytic cathegory ($C^\omega$): an analytic multicritical circle map is an analytic homeomorphism of the circle
with at least one critical point.

A discussion of multicritical circle maps in the real setting has recently been carried out in \cite{GdF}.

In what follows, we will always place one of the critical points of a multicritical circle map at $0$.

 As discussed in some detail in \cite{Ya3}, the space of critical circle maps is
ill-suited to define renormalization. The pioneering works on the subject (\cite{ORSS} and \cite{FKS})
 circumvented this difficulty
by replacing critical circle maps with different objects:
\begin{defn}
A  {\it  $C^r$-smooth (or $C^\omega$) multicritical commuting pair} $\zeta=(\eta,\xi)$ consists of two 
$C^r$-smooth  (or $C^\omega$) orientation preserving interval homeomorphisms 
$\eta:I_\eta\to \eta(I_\eta),\;
\xi:I_{\xi}\to \xi(I_\xi)$, where
\begin{itemize}
\item[(I)]{$I_\eta=[0,\xi(0)],\; I_\xi=[\eta(0),0]$;}
\item[(II)]{Both $\eta$ and $\xi$ have homeomorphic extensions to interval
neighborhoods $V_\eta$, $V_\xi$ of their respective domains {\it with the same degree of
smoothness}, that is $C^r$ (or $C^\omega$), which commute, 
$\eta\circ\xi=\xi\circ\eta$;}
\item[(III)]{$\xi\circ\eta(0)\in I_\eta$;}
\item[(IV)] each of the maps $\eta$ and $\xi$ has finitely many critical points in $V_\eta$, $V_\xi$ respectively, each of them of an odd integer order;
\item[(V)] the point $0$ is a critical point for both $\eta$ and $\xi$.
  \end{itemize}
\end{defn}
We note that the commutation condition (II) ensures that $\eta$ and $\xi$ have the same order of the critical point at $0$.

\begin{figure}[ht]
\caption{\label{compair}A commuting pair}
\includegraphics[width=0.7\textwidth]{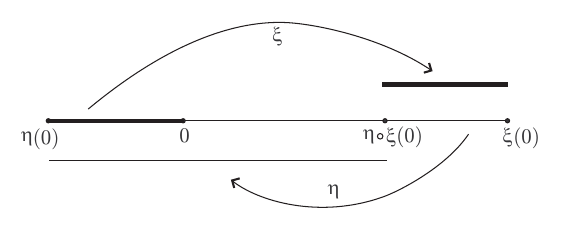}
\end{figure}

Given a multicritical commuting pair $\zeta=(\eta,\xi)$
we can regard the interval $I=[\eta(0),\xi\circ\eta(0)]$ as a circle, identifying 
$\eta(0)$ and $\xi\circ\eta(0)$ and define $F_\zeta:I\to I$ by 
\begin{equation}\label{Fzeta}
F_\zeta=\left\{\begin{array}{l}
                    \eta\circ\xi(x) \text{ for }x\in [\eta(0),0]\\
                    \eta(x)\text{ for } x\in [0,\xi\circ\eta(0)]
\end{array}\right.
\end{equation}
The mapping $\xi$ extends to a $C^r$- (or $C^\omega$-) diffeomorphism of open
neighborhoods of  $\eta(0)$ and $\xi\circ\eta(0)$. Using it as a local chart we turn 
the interval $I$ into a closed one-dimensional manifold $M$. Condition (II) above implies that
the mapping $F_\zeta$ projects to a well-defined $C^3$-smooth homeomorphism $F_\zeta:M\to M$.
Identifying $M$ with the circle by a diffeomorphism $\phi:M\to \TT$
we recover a multicritical circle mapping
\begin{equation}
  \label{fphi}
  f^\phi=\phi\circ F_\zeta\circ\phi^{-1}.
  \end{equation}
The critical circle mappings corresponding to two different choices
of $\phi$ are conjugated by a diffeomorphism, and thus we recovered
a $C^r$- (or $C^\omega$) smooth conjugacy class of circle mappings from a multicritical commuting
pair.

\subsection{Commuting pairs and renormalization of critical circle maps}
The {\it height} $\chi(\zeta)$
of a multicritical commuting pair $\zeta=(\eta,\xi)$ is equal to $r$,
if 
$$0\in [\eta^r(\xi(0)),\eta^{r+1}(\xi(0))].$$
 If no such $r$ exists,
we set $\chi(\zeta)=\infty$, in this case the map $\eta|I_\eta$ has a 
fixed point.  For a pair $\zeta$ with $\chi(\zeta)=r<\infty$ one verifies directly that the
mappings $\eta|[0,\eta^r(\xi(0))]$ and $\eta^r\circ\xi|I_\xi$
again form a commuting pair.

Given a commuting pair $\zeta=(\eta,\xi)$ we will denote by 
$\wtl\zeta$ the pair $(\wtl\eta|\wtl{I_\eta},\wtl\xi|\wtl{I_\xi})$
where tilde  means rescaling by a linear factor $\lambda={1\over |I_\eta|}$.
\begin{defn} We say that a real commuting pair $\zeta=(\eta,\xi)$ is {\it renormalizable} if $\chi(\zeta)<\infty$.
The {\it renormalization} of a renormalizable commuting pair $\zeta=(\eta,
\xi)$ is the commuting pair
\begin{center}
${\cal{R}}\zeta=(
\widetilde{\eta^r\circ\xi}|
 \widetilde{I_{\xi}},\; \widetilde\eta|\widetilde{[0,\eta^r(\xi(0))]}).$
\end{center}
\end{defn}

\noindent
The non-rescaled pair $(\eta^r\circ\xi|I_\xi,\eta|[0,\eta^r(\xi(0))])$ will be referred to as the 
{\it pre-renormalization} $p{\cal R}\zeta$ of the commuting pair $\zeta=(\eta,\xi)$. Suppose $\{\zeta_i\}_{i=1}^{k-1}$ is a sequence of renormalizable pairs such that $\zeta_0=\zeta$ and $\zeta_i=p\cR\zeta_{i-1}$. We call $\zeta_k=p\cR\zeta_{k-1}$ the $k$-th pre-renormalization of $\zeta$; and $\widetilde{\zeta_k}$ the $k$-th  renormalization of $\zeta$ and write
$$\zeta_k=p\cR^k\zeta,\;\widetilde{\zeta_k}=\cR^k\zeta.$$

For a pair $\zeta$ we define its {\it rotation number} $\rho(\zeta)\in[0,1]$ to be equal to the 
continued fraction $[r_0,r_1,\ldots]$ where $r_i=\chi({\cal R}^i\zeta)$. 
In this definition $1/\infty$ is understood as $0$, hence a rotation number is rational
if and only if only finitely many renormalizations of $\zeta$ are defined;
if $\chi(\zeta)=\infty$, $\rho(\zeta)=0$.
Thus defined, the rotation number of a commuting pair can be viewed as a rotation number in
the usual sense:
\begin{prop}
\label{rotation number}
The rotation number of the mapping $F_\zeta$ is equal to $\rho(\zeta)$.
\end{prop}

\noindent
There is an  advantage in defining $\rho(\zeta)$ using a sequence of heights, since it 
removes the ambiguity in prescribing a continued fraction expansion to {\it rational} rotation numbers
in a dynamically natural way.

Let $f$ be a multicritical circle map with a critical point at $0$. Suppose that the continued fraction of $\rho (f)$ contains at least $n+1$ terms; as usual, let $p_k/q_k$ be the $k$-th convergent of $\rho(f)$ for $k\leq n+1$, and let
$$I_n\equiv [0,f^{q_n}(0)].$$

One obtains a multicritical commuting pair
from the pair of maps $(f^{q_{m+1}}|I_m,f^{q_m}|I_{m+1})$ as follows.
Let $\bar f$ be the lift of $f$ to the real line satisfying $\bar f '(0)=0$,
and $0<\bar f (0)<1$. For each $m>0$ let $\bar I_m\subset \Bbb R$ 
denote the closed 
interval adjacent to zero which projects down to the interval $I_m$.
Let $\tau :\Bbb R\to \Bbb R$ denote the translation $x\mapsto x+1$.
Let $\eta :\bar I_m\to \Bbb R$, $\xi:\bar I_{m+1}\to \Bbb R$ be given by
$\eta\equiv \tau^{-p_{m+1}}\circ\bar f^{q_{m+1}}$,
$\xi\equiv \tau^{-p_m}\circ\bar f^{q_m}$. Then the pair of maps
$(\eta|\bar I_m,\xi|\bar I_{m+1})$ forms a multicritical commuting pair
corresponding to $(f^{q_{m+1}}|I_m,f^{q_m}|I_{m+1})$.
Henceforth, we shall abuse notation and simply denote this commuting pair 
by
\begin{equation}
\label{real1}
(f^{q_{m+1}}|I_m,f^{q_m}|I_{m+1}).
\end{equation}
The $m$-th renormalization of $f$ at the critical point $0$ 
is the rescaled pair (\ref{real1}):
\begin{equation}
  \label{real2}
  (\wtl{{f}^{q_{m+1}}}|\wtl{{I}_{m}},\wtl{{f}^{q_{m}}}|\wtl{{I}_{m+1}}).
  \end{equation}

Denote $G(x)$ the Gauss map $$G(x)=\left\{\frac{1}{x}\right\}.$$ It is easy to verify that for a renormalizable pair $\zeta$
\begin{equation}\label{Gauss}
\rho(\cR\zeta)=G(\rho\zeta).
  \end{equation}

\subsection{Dynamical partitions and real {\it a priori} bounds}
Following the convention introduced by Sullivan \cite{Sullivan-bounds}, we say that for an infinitely renormalizable mapping $f$ (or a pair $\zeta$) a quantity is ``beau'' (which translates as ``bounded and eventually universally (bounded)'') if it is bounded for all renormalizations $\cR^n f$ ($\cR^n\zeta)$, and the bound becomes universal (that is, independent of the map) for $n\geq n_0$ for some $n_0\in\NN$. 
\label{sec:partition}
\begin{defn}
  \label{partition}
  Let $f$ be a multicritical circle map which is at least $n$-times renormalizable. The collection of intervals
  $$\cP_n(f)=\{I_n, f(I_n), \ldots, f^{q_{n+1}-1}(I_n)\}\cup \{I_{n+1}, f(I_{n+1}),\ldots, f^{q_n-1}(I_n+1)\}$$
  is called the {\it $n$-th dynamical partition of $f$}.
\end{defn}
The following standard fact is easy to verify:
\begin{prop}
The intervals of $\cP_n(f)$ are disjoint except at the endpoints. Moreover, they cover the whole circle $\TT$.
\end{prop}
Yoccoz \cite{YocDen} demonstrated:
\begin{thm}
  \label{conj-yoc}
  Suppose $\rho(f)\notin \QQ$. Then $$\underset{J\in\cP_n(f)}{\max}|J|\underset{n\to\infty}{\longrightarrow}0.$$
\end{thm}
This implies, in particular:
\begin{thm}
  \label{conj2}

  Suppose $f$ is a muticritical circle map with $\rho(f)\notin\QQ$. Then $f$ is topologically conjugate to the rigid rotation
  $R_{\rho(f)}$.
\end{thm}

A stronger statement is also true  \cite{dFbounds}:
\begin{thm}
  \label{realbounds}
  Suppose $f$ has an irrational rotation number. Then the iterate
  $f^{q_n}|_{I_{n+1}}$ decomposes as
  $$f^{q_n}|_{I_{n+1}}=\psi_{m+1}\circ p_m\circ\psi_{m}\circ p_{m-1}\circ\cdots\circ\psi_1\circ p_0\circ\psi_0$$
  where: \begin{itemize}
  \item $m\leq l+1$, where $l$ is the number of the non-zero critical points of $f$;
  \item $p_j$ is a power law $p_j(x)=x^{b_j}$, $b_j\in 2\ZZ+1$;
    \item $\psi_j$ is an interval diffeomorphism with beau-bounded distortion.
   \end{itemize}
  \end{thm}

Let $\phi$ be a conjugacy $\phi\circ f\circ \phi^{-1}=R_{\rho(f)}$. As before, denote $c_0=0,c_1,\ldots,c_l$ the critical points
of $f$, listed in the counter-clockwise order. Let $k_j$ denote the order of the critical point $c_j$, and set
$$\cK(f)=\{k_0,\ldots,k_l\}.$$
The points $\phi(c_j)$ divide the circle into $l+1$ arcs, let $a_0, a_1,\ldots, a_l$ be their lengths listed in the counter-clockwise order, starting from the point $\phi(0)$
and set
$$\cA(f)=\{a_0,\ldots,a_l\}.$$
We say that  
$\cC(f)=(\rho(f),\cK(f),\cA(f))$ is {\it the marked combinatorial type of $f$}. We say that $f$ and $g$ have {\it the same (unmarked) combinatorial type} if
$\cC(f)=\sigma\cC(g)$, where $\sigma$ is a cyclical permutation, and $\sigma\cC(g)\equiv (\rho(f),\sigma\cK(g),\sigma\cA(g))$.

For a commuting pair $\zeta$ with $\rho(\zeta)\notin \QQ$ we will defined the marked combinatorial type
$$\cC(\zeta)=\cC(f^\phi),\text{ where }f^\phi\text{ is given by  (\ref{fphi}),}$$
and similarly for the unmarked combinatorial type.

To define the $n$-th dynamical partition $\cP_n(\zeta)$ for an $n$-times renormalizable commuting pair $\zeta=(\eta,\xi)$, we simply take the $n$-th dynamical
partition of a map $f^\phi$ as in (\ref{fphi}), and pull it back to $I_\eta\cup I_\xi$.

Both for multicritical circle maps and for commuting pairs we 
will denote by $\overline{\cP_n}$ the set of boundary points of the $n$-th dynamical partition.

Theorem \ref{realbounds} implies:
\begin{thm}
  \label{realbounds2}
Suppose $f$ has an irrational rotation number. Let $J_1$, $J_2$ be two adjacent intervals of the $n$-th dynamical partition $\cP_n(f)$. Then $J_1$, $J_2$ are $K$-commensurable where $K=K(l,\max_{0\leq j\leq l} k_j)$ is beau.
  \end{thm}

\subsection{Periodic renormalization types for bi-critical circle maps}
Let us now specialize to the case when the map $f$ has only two critical points: $0$ and $c$ with orders $k_0$ and $k_1$ respectively. Suppose $\rho(f)$ is irrational. We say that $\rho(f)$ is of periodic type if it is a periodic point of the Gauss map $G$. Let $p$ be the period of $\rho(f)$, and assume that
the orbit of $c_1$ never falls into the critical point $0$. That ensures that for each $j\in\NN$ the pair of maps $p\cR^j f$ has exactly two critical points: $0$ and $c_j\neq 0$ with orders $k_0$, $k_1$.

\begin{defn}
We say that $f$ is of a {\it periodic marked combinatorial type} with period $n$, where $p|n$, if 
$\cC(f)=\cC(R^n f)$. 
\end{defn}

An equivalent way of defining a periodic marked combinatorial type uses dynamical partitions as follows; it will be useful in the applications.
\begin{defn}
Consider the $2$-nd dynamical partition of the renormalization $\cR^{j+m}(f)$ for $0\leq m\leq p-1$. It has $v_m$ intervals, let us enumerate them
$P_1,\ldots,P_{v_m}$ from left to right. Let $i=w_{j,m}$ denote the number of the interval $P_i$ which contains $c_{j+m}$.
  We say that $f$ is of {\it a periodic marked combinatorial type} if:
  \begin{enumerate}
  \item  the numbers 
\begin{equation}\label{eq:type}
  w_m\equiv w_{j,m}
\end{equation}
are independent of $j$;
\item   there exists $N\leq m$ such that for any $j$ the two critical points of $\cR^j f$ are separated by an interval of $\cP_N(\cR^j f)$.
\end{enumerate}
  \end{defn}
The equivalence of the two definitions is a straightforward consequence of  Theorem~\ref{conj-yoc}, and will be left to the reader.

Another straightforward consequence of  Theorem~\ref{conj-yoc} is the following:
  \begin{prop}
    \label{prop-type}
Suppose $f$ and $g$ both are of a periodic marked combinatorial type. Suppose $\rho(f)=\rho(g)$ and $f$ and $g$ have the same values of $w_m$ (\ref{eq:type}). Then 
$$\cC(f)=\cC(g).$$
  \end{prop}

  \subsection{Renormalization theorems for bi-cubic critical circle maps}
  In the following two statements we assume that the critical circle map has two critical points, both of which are of order $3$. We call such maps {\it bi-cubic}.

  \begin{thm}
    \label{main1}
    Let $\cC$ be a periodic marked combinatorial type with period $m$. Then there exists a commuting pair $\zeta_*\in C^\omega$ such that the following properties hold:
    \begin{enumerate}
    \item $\cC(\zeta_*)=\cC$;
      \item $\cR^m(\zeta_*)=\zeta_*$;
      \item for all $\zeta\in C^\omega$ such that $\cC(\zeta)=\cC$ we have
        $$\cR^{mj}\zeta\underset{j\to\infty}{\longrightarrow}\zeta_*$$
        in the uniform sense.
    \end{enumerate}
 \end{thm}   

  The next theorem talks about the hyperbolicity properties of the renormalization operator at the periodic point $\zeta_*$. We will see in section \ref{section: cyl ren} how to make a {\it canonical} choice of the analytic diffeomorphism $\phi=\phi_\zeta$ in (\ref{fphi}), so that the correspondence
  \begin{equation}
    \label{eq-cylren1}
    f\mapsto \zeta\equiv \cR^N f\mapsto \phi_\zeta\circ F_\zeta\circ \phi_\zeta^{-1}
    \end{equation}
  becomes a well-defined renormalization operator on the space of bi-cubic circle maps, rather than commuting pairs.  
We will then show:
  
  \begin{thm}
    \label{main2a} 
    There exists a Banach manifold structure ${\mathbb C}_U$ on the space of analytic  bi-cubic critical circle maps, such that the following holds.
    The Banach structure is compatible with the topology of uniform convergence. The renormalization operator defined by (\ref{eq-cylren1}) is an analytic operator, defined on an open set in this Banach manifold. The point $f_*=\phi_{\zeta_*}\circ\zeta_*\circ\phi_{\zeta_*}^{-1}$ is a hyperbolic periodic point of this operator, with two unstable eigendirections.
    \end{thm}

\section{Holomorphic commuting pairs}\label{sec:holpairs}
Below we define holomorphic extensions of analytic commuting pairs which generalize de~Faria's definition from \cite{dF2}
for the case of multicritical circle maps.
We say that
a multicritical commuting pair $\zeta=(\eta|_{I_\eta},\xi|_{I_\xi})$  extends to a {\it
holomorphic commuting
pair} $\cH$ if there exist 
three simply-connected $\RR$-symmetric domains $D, U, V\subset\CC$ whose intersections with the real line are denoted by $I_U=U\cap\RR$, $I_V=V\cap\RR$, $I_D=D\cap\RR$, and a simply connected $\RR$-symmetric Jordan domain $\Delta$, such that

\begin{itemize}
\item the endpoints of $I_U$, $I_V$ are critical points of $\eta$, $\xi$ respectively;

\item $\bar D,\; \bar U,\; \bar V\subset \Delta$;
 $\bar U\cap \bar V=\{ 0\}\subset D$; the sets
  $U\setminus D$,  $V\setminus D$, $D\setminus U$, and $D\setminus V$ 
  are nonempty, connected, and simply-connected; 
  $I_\eta\subset I_U\cup\{0\}$, $I_\xi\subset I_V\cup\{0\}$;
\item  the sets $U\cap\HH, V\cap\HH, D\cap\HH$ are Jordan domains;

\item the maps $\eta$ and $\xi$ have analytic extensions to $U$ and $V$ respectively, so that $\eta$ is a branched covering map of $U$ onto $(\Delta\setminus\RR)\cup\eta(I_U)$, and $\xi$ is a branched covering map of $V$ onto $(\Delta\setminus\RR)\cup\xi(I_V)$, with all the critical points of both maps contained in the real line;

\item the maps $\eta\colon U\to\Delta$ and $\xi\colon V\to\Delta$ can be further extended to analytic maps $\hat{\eta}\colon U\cup D\to\Delta$ and $\hat{\xi}\colon V\cup D\to\Delta$, so that the map $\nu=\hat{\eta}\circ\hat{\xi}=\hat{\xi}\circ\hat{\eta}$ is defined in $D$ and is a  branched covering of $D$ onto $(\Delta\setminus \RR)\cup\nu(I_D)$ with only real branched points.

\end{itemize}

\noindent
We shall identify a holomorphic pair $\cH$ with a triple of maps $\cH=(\eta,\xi, \nu)$, where $\eta\colon U\to\Delta$, $\xi\colon V\to\Delta$ and $\nu\colon D\to\Delta$. We shall also call $\zeta$ the {\it commuting pair underlying $\cH$}, and write $\zeta\equiv \zeta_\cH$. When no confusion is possible, we will use the same letters $\eta$ and $\xi$ to denote both the maps of the commuting pair $\zeta_\cH$ and their analytic extensions to the corresponding domains $U$ and $V$.

The sets $\Omega_\cH=D\cup U\cup V$ and $\Delta\equiv\Delta_\cH$ will be called \textit{the domain} and \textit{the range} of a holomorphic pair $\cH$. We will sometimes write $\Omega$ instead of $\Omega_\cH$, when this does not cause any confusion. 

We can associate to a holomorphic pair $\cH$  a piecewise defined map $S_\cH\colon\Omega\to\Delta$:
\begin{equation*}
S_\cH(z)=\begin{cases}
\eta(z),&\text{ if } z\in U,\\
\xi(z),&\text{ if } z\in V,\\
\nu(z),&\text{ if } z\in\Omega\setminus(U\cup V).
\end{cases}
\end{equation*}
De~Faria \cite{dF2} calls $S_\cH$ the {\it shadow} of the holomorphic pair $\cH$.

We can naturally view a holomorphic pair $\cH$ as three triples
$$(U,\xi(0),\eta),\;(V,\eta(0),\xi),\;(D,0,\nu).$$
We say that a sequence of holomorphic pairs converges in the sense of Carath{\'e}odory convergence, if the corresponding triples do.
We denote the space of triples equipped with this notion of convergence by $\hol$.

We let the modulus of a holomorphic commuting pair $\cH$, which we denote by $\mod(\cH)$ to be the modulus of the largest annulus $A\subset \Delta$,
which separates $\CC\setminus\Delta$ from $\overline\Omega$.

\begin{defn}\label{H_mu_def}
For $\mu\in(0,1)$ let $\hol(\mu)\subset\hol$ denote the space of holomorphic commuting pairs
${\cH}:\Omega_{{\cH}}\to \Delta_{\cH}$, with the following properties:
\begin{enumerate}
\item $\mod (\cH)\ge\mu$;
\item {$|I_\eta|=1$, $|I_\xi|\ge\mu$} and $|\eta^{-1}(0)|\ge\mu$; 
\item $\dist(\eta(0),\partial V_\cH)/\diam V_\cH\ge\mu$ and $\dist(\xi(0),\partial U_\cH)/\diam U_\cH\ge\mu$;
\item {the domains $\Delta_\cH$, $U_\cH\cap\HH$, $V_\cH\cap\HH$ and $D_\cH\cap\HH$ are $(1/\mu)$-quasidisks.}
\item$\diam(\Delta_{\cH})\le 1/\mu$;
\end{enumerate}
\end{defn}

Let the {\it degree} of a holomorphic pair $\cH$ denote the maximal topological degree of the covering maps constituting the pair. Denote $\hol^K(\mu)$ the subset of $\hol(\mu)$ consisting of pairs whose degree is bounded by $K$. The following is an easy generalization of Lemma 2.17 of \cite{Ya4}:  

\begin{lem}
\label{bounds compactness}
For each $K\geq 3$ and $\mu\in(0,1)$ the space $\hol^K(\mu)$ is sequentially compact.
\end{lem}

\noindent
We say that a real commuting pair $\zeta=(\eta,\xi)$ with
an irrational rotation number has
{\it complex {\rm a priori} bounds}, if there exists $\mu>0$ such that all renormalizations of $\zeta=(\eta,\xi)$ extend to 
holomorphic commuting pairs in $\hol(\mu)$.
The existense of complex {\it a priori} bounds is a key analytic issue 
 of renormalization theory. 

\begin{defn}\label{A_r_com_pair_def}
For a set $S\subset\CC$, and $r>0$, we let $N_r(S)$ stand for the $r$-neighborhood of $S$ in $\CC$. 
For each $r>0$ we introduce a class $\cA_r$ consisting of pairs   
$(\eta,\xi)$ such that the following holds:
\begin{itemize}
\item $\eta$, $\xi$ are real-symmetric analytic maps defined in the domains 
$$N_r([0,1])\text{ and }N_{r|\eta(0)|}([0,\eta(0)])$$
respectively, and continuous up to the boundary of the corresponding domains;
\item the pair $$\zeta\equiv (\eta|_{[0,1]},\xi|_{[0,\eta(0)]})$$
is a multicritical commuting pair.
\end{itemize}
\end{defn}

\noindent
For simplicity, if $\zeta$ is as above, we will write $\zeta\in\cA_r$. But it is important to note that viewing our multicritical commuting pair $\zeta$ as an element of $\cA_r$ imposes restrictions on where we are allowed to iterate it. Specifically, we view such $\zeta$ as undefined at any point $z\notin N_r([0,\xi(0)])\cup N_r([0,\eta(0)])$ (even if $\zeta$ can be analytically continued to $z$). Similarly, when we talk about iterates of $\zeta\in\cA_r$ we iterate the restrictions $\eta|_{N_r([0,\xi(0)])}$ and $\xi|_{N_{r|\eta(0)|}([0,\eta(0)]}$.
In particular, we say that the first and second elements of $p\cR\zeta=(\eta^r\circ\xi,\eta)$ are defined in the maximal domains, where the corresponding iterates are defined in the above sense.

For a domain $\Omega\subset\CC$, we denote by $\mathfrak D(\Omega)$ the complex Banach space of analytic functions in $\Omega$, continuous up to the boundary of $\Omega$ and equipped with the sup norm. Consider the mapping $i\colon\cA_r\to\mathfrak D(N_r([0,1]))\times\mathfrak D(N_r([0,1]))$ defined by 
\begin{equation}\label{norm_inclusion_eq}
i(\eta,\xi)=(\eta, h\circ\xi\circ h^{-1}),\quad\text{ where }\quad h(z)=z/\eta(0).
\end{equation}
The map $i$ is injective, since $\eta(0)$ completely determins the rescaling $h$. Hence, this map induces a metric on $\cA_r$ from the direct product of sup-norms on $\mathfrak D(N_r([0,1]))\times\mathfrak D(N_r([0,1]))$. This metric on $\cA_r$ will be denoted by $\dist_r(\cdot,\cdot)$.

We will denote $\cA_r^L$ the subset of $\cA_r$ consisting of pairs whose degree is bounded by $L$.

\begin{thm}[{\bf Complex bounds for periodic type}]
  \label{thm:bounds}
Let $\cC$ be a periodic combinatorial type and $L\geq 3$. 
  There exists a constant $\mu>0$  such that
the following holds. For every positive real number $r>0$ and every pre-compact family $S\subset\mathcal A_r^L$ of multicritical commuting pairs, there exists $N=N(r, S)\in\NN$ such that if $\zeta\in S$ is
a commuting pair with  $\cC(\zeta)=\cC$, then
$p\cR^n\zeta$ restricts to a holomorphic commuting pair $\cH_n:\Omega_n\to\Delta_n$ with $\Delta_n\subset N_r(I_\eta)\cup N_r(I_\xi)$. 
Furthermore, the range $\Delta_n$ is a Euclidean disk,  and the appropriate affine rescaling of $\cH_n$ is in $\hol(\mu)$.
\end{thm}

\noindent
In the next section we will give a proof of this theorem which generalizes our proof in \cite{Ya1}. For simplicity of notation, we will assume that the pair $\zeta$ is a pre-renormalization of a multicritical circle map $f$, that is
$$\zeta=(f^{q_n},f^{q_{n+1}}).$$
This will allow us to write explicit formulas for long compositions of terms $\eta$ and $\xi$, which will greatly streamline the exposition.
Theorem~\ref{thm:bounds} follows from Theorem~\ref{thm:power} which we state and prove in the following section.

\section{Complex {\it a priori} bounds}
\subsection{Power law estimate}

Let $f$ be an analytic multicritical circle map with a periodic combinatorial type $\cC$ as above. Let $U$ be a $\TT$-symmetric annulus which is contained in the domain of analyticity of $f$. 
Suppose $0$ is a critical point of $f$ of an odd integer order $k=2m+1$ and denote $\cR^n f$ the $n$-th renormalization of $ f$ at $0$ as above. Note that in what follows, rather than working with the map of the circle $f$, we will work with the appropriately translated lifts of $f$ (\ref{real1}), so all of our estimates will be for maps defined in a neighborhood of the real line rather than $\TT$.

Given an interval $ J\subset \Bbb R$, let
$\CC_{J}\equiv \Bbb C\backslash (\Bbb R\backslash {\rm J})$ denote
the plane slit along two rays. Let  $\overline{{ {\Bbb C}}_J}$
 denote the completion of
this domain in the path metric in ${\Bbb C}_{J}$ (which means that we add to
 ${\Bbb C}_{J}$
 the banks of the slits).

The following theorem is a generalization of the main estimate of \cite{Ya1}:
\begin{thm}\label{thm:power}
  There exist  positive constants $R_0, b, c$ depending only on $\cC$ such that the following holds. Let $f$ be as above, and $R\geq R_0$. Then there exists $n_0\in\NN$ which depends on $f$ and $R$ such that for all $n\geq n_0$ we have the following.

  Denote $$\cR^n(f)\equiv(\eta,\xi).$$
  There exist subdomains
  $V_\eta\supset \overset{\circ}{I}_\eta$, $V_\xi\supset\overset{\circ}{I}_\xi$ such that the map
  $\eta$ is a branched covering $$V_\eta\longrightarrow \DD_R\cap \CC_{\eta(I_\eta)},$$
  and similarly for $\xi$. Furthermore, for any $z\in V_\eta\cup V_\xi$ we have 
\begin{equation}
\label{key estimate}
|\cR^nf(z)|\geq c|z|^k+b,
\end{equation}
where the left-hand side stands for $\eta$ on $V_\eta$ and $\xi$ on $V_\xi$.

Finally, the constant $n_0$ can be chosen to depend only on $R$ in any family of maps which is pre-compact in the compact-open topology on $U$.
\end{thm}

Fix an integer $M>0$ and $n\geq M$. 
Let
$$H_m\equiv [f^{q_{m+1}}(0),f^{q_{m}-q_{m+1}}(0)],$$
and let $D_m$ denote the Euclidean disc 
$D(H_m)$.
Assume that $D_M\subset U$.
 Consider the  inverse orbit:
\begin{equation}
\label{J-orbit}
J_0\equiv f^{q_{n+1}}(I_n),J_{-1}\equiv f^{q_{n+1}-1}(I_n),\ldots,
J_{-(q_{n+1}-1)}\equiv f(I_n).
\end{equation}
For a point $z\in D_M$ we say that
\begin{equation}
\label{z-orbit}
z_0\equiv z,z_{-1},\ldots,
z_{-(q_{n+1}-1)}
\end{equation}
is a {\it corresponding} inverse orbit if each $z_{-(k+1)}$ is obtained by applying to $z_{-k}$ a univalent inverse branch of $f|_U$ where
$U$ is a sub-interval of $J_{-(k+1)}$.

We derive Theorem~\ref{thm:power} from the following assertion:
\begin{lem}
\label{lin}
There exists $M>0$ which depends only on $f$, and can be chosen uniformly in a compact family, such that for all $n>M$ 
 and any $z \in\C{f^{q_{n+1}}(I_n)}\cap D_M$ the following estimate holds:
\begin{equation}
\label{linear}  
\frac{\dist(z_{-(q_{n+1}-1)},f(I_n))}{|f(I_n)|}\leq C
\frac{\dist(z,I_n)}
{|I_n|},
\end{equation}
where $z_n$ is as in (\ref{z-orbit}) and $C$ is a universal constant.
\end{lem}

\subsection{Poincar{\'e} neighborhoods}
By symmetry, $J$ is a hyperbolic geodesic in ${\Bbb C}_{J}$.
Consider the {\it hyperbolic neighborhood} of $J$  of radius
$r$ that is the 
set of all points in ${\Bbb C}_{J}$
 whose hyperbolic distance to $J$ 
 is less than $r$. One verifies directly that a hyperbolic neighborhood is
the union of two Euclidean discs symmetric to each other with respect
to the real axis, with a common chord $J$.
We will  denote such a neighborhood $D_\theta(J)$, where $\theta=\theta(r)$
is the outer angle the boundaries of the two discs form with the real
axis. An elementary computation yields:
\begin{equation}
  \label{rtheta}
  r=\log\cot(\theta/4).
\end{equation}
Note,  that in particular the Euclidean disc $D(J)\equiv D_{\pi/2}(J)$
can  be interpreted as a hyperbolic neighborhood of $J$.
These hyperbolic neighborhoods were introduced into the subject by Sullivan
\cite{Sullivan-bounds}. They are a key tool for proving complex bounds
due to the following version of the Schwarz Lemma:

\begin{lem}
\label{lem:inv}
  Let us consider two intervals  
$J'\subset J\subset {\Bbb R}$. 
Let $\phi:{\Bbb C}_{J}\rightarrow{\Bbb C}_{J'}$ be an analytic map such that
$\phi(J)\subset J'$. Then for any $\theta\in (0, \pi),$ 
$\phi (D_\theta (J))\subset D_\theta (J')$.
\end{lem}

\begin{figure}[ht]
  \caption{\label{fig:invariance}An illustration of Lemma~\ref{lem:inv}}
  \centerline{\includegraphics[width=0.65\textwidth]{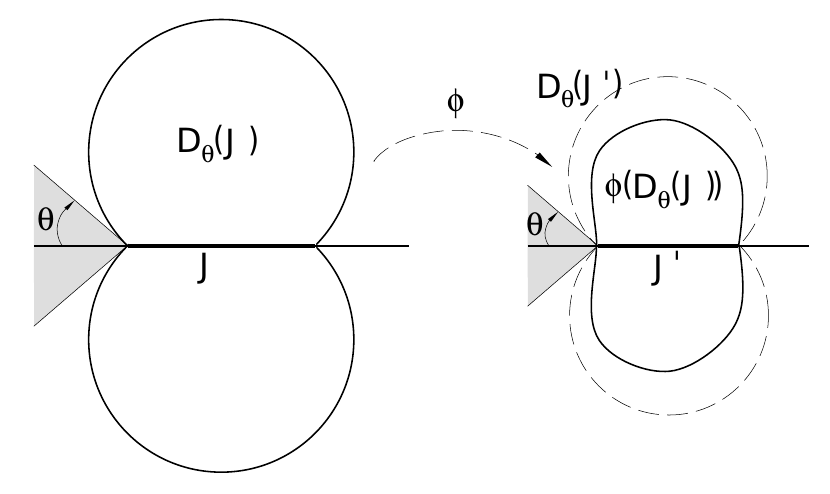}}
\end{figure}

\noindent
The above invariance statement becomes quasi-invariance if the domain of $\phi$ is not the whole double-slit plane, yet is large enough relative to the size of $J$. Let us state this as a Lemma below. Denote $\DD_R=\{|z|<R\}$.

\begin{lem}
  \label{quasi1}
  For every sufficiently large $R>0$ there exists $\theta(R)>0$ with $\theta(R)\to 0$ and $(R\theta(R))^{-1}\to 0$ as $R\to\infty$ such that the following holds. Let
  $$\phi:\DD_R\cap \CC_{[-1,1]}\to\CC$$ be a univalent and real-symmetric function satisfying $\phi(\pm 1)=\pm 1$. Then for every $\theta\geq \theta(R)$ we have
  $$\phi(D_\theta([-1,1]))\subset D_{(1-R^{-1-\delta})\theta}([-1,1]),$$
  where $\delta\in(0,1)$ is a fixed constant.

\end{lem}
While superficially similar to \cite[Lemma~2.4]{dFdM2}, it is a much stronger statement: we do not assume that $\phi$ is conformal in all of $\DD_R$ (since this will clearly not be the case in our applications).
Before proving this statement, let us formulate a lemma (see \cite[Lemma3]{Wars}):
\begin{lem}
  \label{lem:war}
  Suppose $\Omega$ is a Jordan subdomain of $\DD$ such that
  $$\partial\Omega\subset \{|z|\geq 1-\eps\}\text{ where }\eps>\frac{1}{2}.$$
  There exists $C>0$ such that the following holds.
  Let $h:\Omega\to\DD$ be the Riemann mapping normalized by $h(0)=0$ and $h'(0)>0$. Then
  $$|h^{-1}(z)-z|<C\eps,$$
  for all $|z|<1/2$.
\end{lem}
Combining this with the Cauchy Integral Formula, we obtain
\begin{equation}
  \label{eq:wars}
  g'(0)<1+D\eps,
\end{equation}
where $D$ is a universal constant.

\begin{proof}[Proof of Lemma \ref{quasi1}]
  We start with an exercise in conformal mapping. Let $G(z)$ denote the branch of $\sqrt{z}$ which is defined in $\CC\setminus \RR_{<0}$ and preserves $\RR_{>0}$; and let
  $$F(z)=\frac{1-z}{1+z}.$$
  Then, the composition $\Psi\equiv -F\circ G\circ F$ is a conformal mapping of $\CC_{[-1,1]}$ onto $\DD$ which preserves the interval $[-1,1]$.
  Easy estimates show that the image of $\partial\DD_R$ under $\Psi$ is contained in $\eps$-neighborhoods of $\pm i$ with $\eps=O(R^{-1})$.
  Let us denote $$\Omega\equiv \Psi(\DD_R\cap\CC_{[-1,1]}).$$
  The hyperbolic metric of $\Omega$ at a point $a$ is obtained by the pull-back of the Euclidean metric by the composition
  $h_a\circ M_a$ where $M_a$ is the M{\"o}bius map $(z-a)/(1-\bar a z)$ and $h_a$ is the conformal map of $\Omega_a\equiv M_a(\Omega)$ onto the unit disk
  with $h_a(0)=0$, $h_a'(0)>0$.
  
  Fix $K>0$. For $z$ with $\dist_{\bar \DD}(z,[-1,1])<K$,
  $$\partial\Omega_a\subset \{|z|>\eps(a)\}\text{ with }\eps(a)=O(R^{-1})$$
  (this estimate could be much improved for $z$ near $\pm 1$, but we will not require it).
  Using (\ref{eq:wars}), we see that $K$-neighborhood of $[-1,1]$ in the hyperbolic metric of $\Omega$ is contained in the $K'$-neighborhood of $[-1,1]$ in the
  hyperbolic metric of $\DD$ with
  $$K'=K(1+O(R^{-1})).$$
    Together with \ref{rtheta} and the Schwarz Lemma, this implies the claim.

\end{proof}
Easy induction combined with real {\it a priori} bounds implies (cf. \cite[Lemma 2.5]{dFdM2}):
\begin{lem}
  \label{quasi2}
  For each $n\geq 1$ there exist $K_n\geq 1$ and $\theta_n>0$ with $K_n\to 1$ and $\theta_n\to 0$ as $n\to\infty$ such that the following holds. Let $\theta\geq \theta_n$, and let  $0\leq i<j\leq q_{n+1}$ be such that the restriction
  $$f^{j-i}:f^i(I_n)\to f^j(I_n)$$
  is a diffeomorphism on the interior. Then the inverse branch $f^{-(j-i)}|_{f^j(I_n)}$ is well-defined over $D_\theta(f^j(I_n))$ and maps it univalently into the
 Poincar{\'e}  neighborhood $D_{\theta/K_n}(f^i(I_n))$.
\end{lem}

\subsection{Proof of complex {\it a priori} bounds}
Let $J=[a,b]$, $a<b$.  
For a point $z\in {\overline{\Bbb C}}_J$, the 
{\it angle between $z$ and 
 $J$}, $\ang{z}{J}$
is the least of the angles  between the intervals $[a,z]$, $[ b, z]$  and 
the corresponding rays $(a,-\infty]$, $[b, +\infty)$ of the real line,
    measured in the range $ 0\leq \theta\leq \pi$. In what follows, a beau bound $\eps>0$ on the angle will appear; we will say that
    a point $z_{-i}$ of (\ref{z-orbit}) {\it $\eps$-jumps} if $\ang{z_{-i}}{J_{-i}}>\eps$. We will also say that such points have a {\it good angle}.

The next statement shows that for points with a good angle the normalized distance to the intervals of (\ref{J-orbit}) cannot increase by more than a fixed multiple. It is taken from \cite{Ya1}. However, its proof is somewhat modified, since for a multicritical circle map
the orbit (\ref{J-orbit}) may pass through 
critical values. 

\begin{lem}\label{good angle} We work with the notations of Lemma~\ref{lin}. Fix $\eps_1>0$, $B>0$. Let $J=J_{-i}$, $J'=J_{-k}$ be two intervals of the inverse orbit (\ref{J-orbit}) for $k>i$, and let $z$, $z'$ be the corresponding points of (\ref{z-orbit}).
Assume that
$\ang{z}{J}\geq\epsilon_1$ and $\dist(z,J)\geq B|J|$. Then
$${\dist(z', J')\over |J'|}\leq C{\dist(z, J)\over |J|}$$ 
for some constant $C=C(\epsilon_1,B)$
\end{lem}

\begin{pf} 
Suppose, $f$ has $l$ critical points, not counting $0$.
Since the orbit (\ref{J-orbit}) forms a part of the dynamical partition of the circle of level $n$, there are at most $l$ critical points
for the iterate $f^{k-i}|_{J'}$. Hence, there exists a sub-interval $K'\subset J'$ which is commensurable with $J'$ such that
$f^{k-i}:{K'}\mapsto K$ is a diffeomorphism. 
By the same reasoning as in \cite{Ya1},  the smallest closed hyperbolic neighborhood 
$\cl D_{\theta}(K)$ enclosing $z$ satisfies 
$\diam D_{\theta}(K)\leq C(\epsilon,B)\dist (z, K)$.
The claim now follows by Lemma~\ref{quasi2}.
\end{pf}

We now formulate an analogue of Lemma~4.2 of \cite{Ya1}.
\begin{lem}
\label{cb1}
Let $J\equiv J_{-k},J_{-k-q_{m+1}}\equiv J'$ be two consecutive returns of the 
backward orbit (\ref{J-orbit}) to $I_m$, for $m\geq M$, and let $\zeta$ and
$\zeta'$ be the corresponding points of the orbit (\ref{z-orbit}).
Suppose $\zeta\in D_m$, then 
either $\zeta'\in D_m$, or
$\ang{\zeta'}{J'}>\eps$, and $\dist({\zeta'},{J'})<C|I_m|$,
where the quantifiers $\eps$ and $C$ are beau.
\end{lem}
\begin{pf}
  The proof is a very slight modification of the proof of Lemma~4.2 of \cite{Ya1}, the only difference being that the iterate $f^{q_{m+1}}$ restricted to $H_m$ may have at most $l$ critical points in the interior.
  We refere reader to figure~\ref{fig:bounds1} for an illustration.

  \begin{figure}[ht]
    \caption{\label{fig:bounds1}An illustration to the proof of Lemma \ref{cb1}}
    \includegraphics[width=\textwidth]{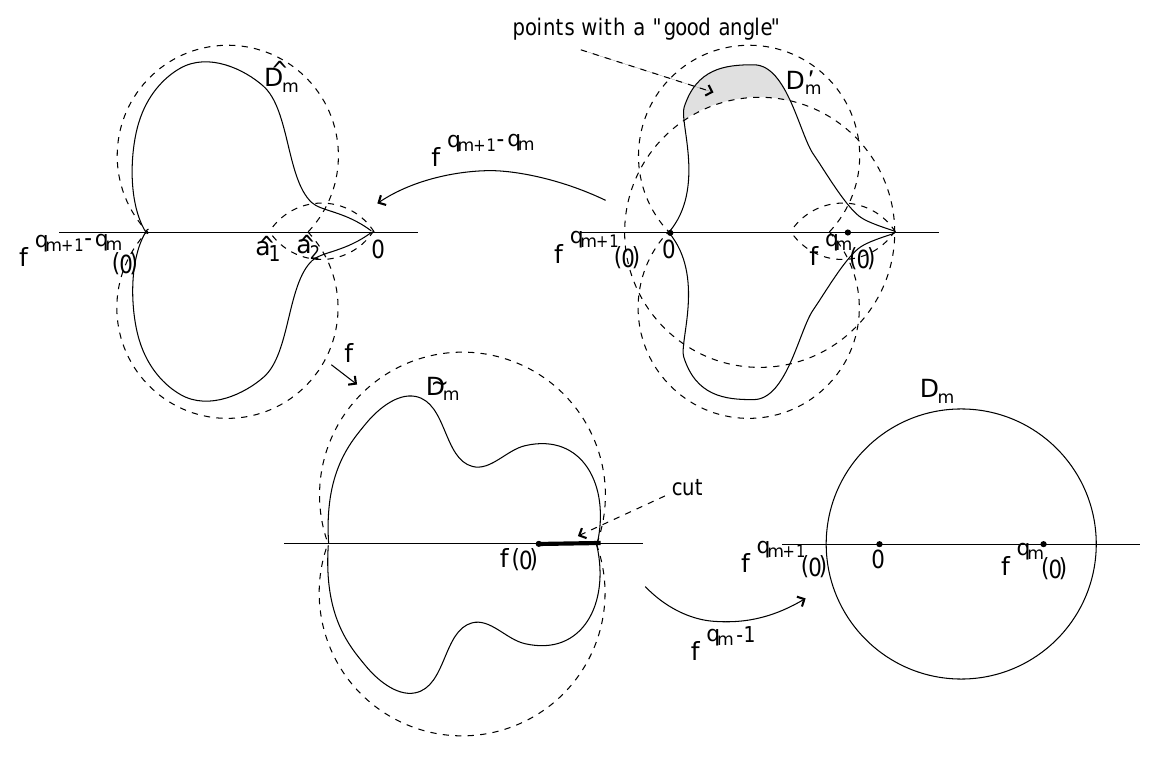}
   \end{figure}
  
  Let $$\tl H_m=f^{-(q_m-1)}(H_m)$$
  and let $\tl D_m$ be the connected component of $f^{-(q_m-1)}(D_m)$ which intersects
  the real line over the interior of $\tl H_m$. Note that the iterate $f^{q_m-1}$
  has at most $l$ critical points in $\tl H_m$ and they divide it into commensurable parts. Hence, by Lemma~\ref{quasi2}, the domain $\tl D_m\subset D_{\theta_1}(\tl H_m)$ where $\theta_1$ is beau. 

  To fix the ideas, suppose that $f^{-q_m+1}(0)<f(0)$. Let $\hat H_m$ denote the interval $[f^{q_{m+1}}(0),0]$ and let $\hat D_m$ be the 
  domain which intersects the real line over the interior of $\hat H_m$ and such that $$f:\hat D_m\to\tl D_m\setminus \RR_{\geq f(0)}$$ is a branched covering (this is a precise way of saying that $\hat D_m$ is a pull-back of $\tl D_m$ by the ``branch'' of $f^{-1}$ defined in $\CC_{[f^{q_{m+1}-q_m+1}(0),f(0)]}$).
  Since $f$ has a critical point of order $k\geq 3$ at $0$, the boundary of $\hat D_m$ forms angles $\pi/k\leq\pi/3$ with the negative real axis. In particular, there exist $\tl a_1$, $\tl a_2$ in $\hat H_m$ which divide it in commensurable parts, so that
  $$\hat D_m\subset D_{\theta_2}([f^{q_{m+1}-q_m}(0),\tl a_2])\cup D_{5\pi/8}([\tl a_1, 0]),$$
  where $\theta_2$ is beau.

  Again using Lemma~\ref{quasi2} for the pull-back $D_m'$ of $\hat D_m$ by the iterate
  $f^{q_{m+1}-q_m}$ along the real line, we see that there exist $a_1$, $a_2$ which divide $[0,f^{q_m-q_{m+1}}(0)]$ into commensurable sub-intervals such that
  $$D'_m\subset D_{\theta_3}([0,a_2])\cup D_{13\pi/24}([a_1,f^{q_{m}-q_{m+1}}(0)]),$$
  where $\theta_3$ is beau, and the proof is completed.

\end{pf}

\noindent
The next statement is a direct analogue of Lemma 4.4 of \cite{Ya1}. The proof from \cite{Ya1} is again adapted {\it mutatis mutandis} to
our situation, in the same way as the preceding proof. We will leave the details to the reader.

\begin{lem}
\label{cb2}
Let $J$ be  the last return of the backward  orbit (\ref{J-orbit})
to the interval $I_m$ before the first return to $I_{m+1}$,
and let $J'$ and $J''$ be the first two returns of (\ref{J-orbit})
 to $I_{m+1}$. Let $\zeta$, $\zeta'$, $\zeta''$ be the corresponding moments
in the backward orbit (\ref{z-orbit}),  $\zeta =f^{q_m}(\zeta')$, 
$\zeta' =f^{q_{m+2}}(\zeta'')$.

Suppose $\zeta\in D_m$. Then either 
$\ang{\zeta''}{I_{m+1}}>\eps$ and $\dist(\zeta'',J'')<
C|I_{m+1}|$, or $\zeta''\in D_{m+1}$; where the constants $C$ and $\eps$ are beau.
\end{lem}

\begin{proof}[Proof of Lemma \ref{lin}]
We start with a point $z\in D_M$. Consider the largest $m$ such that
$D_m$ contains $z$. We will carry out induction in $m$.
Let $P_0,\dots,P_{-k}$ be the consecutive returns of the backward orbit
(\ref{J-orbit}) to the interval $I_m$ until the first return to $I_{m+1}$,
and denote by $z=\zeta_0,\dots,\zeta_{-k}=\zeta'$ the corresponding points
of the orbit (\ref{z-orbit}). Note that the intervals $P_{-i}$ are commensurable with $J_0$.
By \lemref{cb1}  $\zeta_{-i}$ either $\eps$-jumps
and 
$\dist(\zeta_{-i},P_{-i})\leq C|I_m|$, or 
$\zeta'\in D_m$.

 In the former case we are done by \lemref{good angle}.
In the latter case consider the point $\zeta''$ which corresponds to the
second return of the orbit (\ref{J-orbit}) to $I_{m+1}$. By \lemref{cb2},
either $\ang{\zeta''}{I_{m+1}}>\eps$, and $\dist({\zeta''},{I_{m+1}})\leq
C|I_{m+1}|$, or $\zeta''\in D_{m+1}$.

In the first case we are done again by \lemref{good angle}. In the 
second case, the argument is completed by induction in $m$.

\end{proof}

\section{Proof of Theorem \ref{main1}}
\label{sec:proofmain1}
\subsection{Quasiconformal conjugacies between holomorphic pairs}

Let us first quote the following (see \cite{Herm,dFbounds}):
\begin{thm}
  \label{qsconj}
  Suppose $\zeta_1$, $\zeta_2$ are two commuting pairs with
  $$\rho(\zeta_1)=\rho(\zeta_2)\notin\QQ\text{ and }\cC(\zeta_1)=\cC(\zeta_2).$$
  Then for $n\geq 1$, the renormalizations $\cR^n\zeta_i$ are $K_n$-quasi-symmetrically conjugate, and the constant $K_n$ is beau.
\end{thm}

Combining this  with Theorem~\ref{thm:bounds} and using the standard pull-back argument (see e.g. \cite{dF2,Ya1}), we obtain:
\begin{thm}
  \label{qcconj}
Let $\cC$ be a periodic combinatorial type and $L\geq 3$. 
  There exists a constant $K>0$  such that
  the following holds. For every positive real number $r>0$ and every pre-compact family $S\subset\mathcal A_r^L$ of critical commuting pairs, there exists $N=N(r, S)\in\NN$ and $K>1$ such that the following holds. Let $\zeta_1,\zeta_2\in S$ be two commuting pairs with
  $$\rho(\zeta_1)=\rho(\zeta_2)\notin\QQ\text{ and }\cC(\zeta_1)=\cC(\zeta_2)=\cC,$$
  and $\cH_n^i$ is the holomoprhic commuting pair extension of $p\cR^n\zeta_i$ constructed in Theorem~\ref{thm:bounds}, then
  $\cH_n^1$ is $K$-quasiconformally conjugate to $\cH_n^2$.
\end{thm}
For bi-cubic critical circle maps we can refine this statement further by showing that the quasiconformal conjugacy is actually conformal on the filled Julia set:
\begin{thm}
  \label{hybrid}
  Suppose $f$ is a bi-cubic critical circle map, $\rho(f)\notin\QQ$, and let $\cH$ be a holomorphic pair extension of $\cR^n f$. Then any Beltrami differential $\mu$ which is $\RR$-symmetric, $\cH$-invariant, and with support $\text{supp}(\mu)\subset K(\cH)$ is necessarily trivial:
  $$\mu\underset{a.e.}{=}0.$$
\end{thm}  
Theorem \ref{main1} follows from the two above statements and McMullen's Tower Rigidity argument (see \cite{dFdM2,Ya4}).
The proof of Theorem~\ref{hybrid} will occupy the rest of this section.

\subsection{Zakeri's models for bi-cubic circle maps}
Note that for a bi-cubic critical circle map $f$ the combinatorial type $\cC(f)$ is expressed by a single number.
Consider the family of degree $5$ Blaschke products
\begin{equation}
  \label{eq:blaschke}
B(z)=e^{2\pi it}z^3\left( \frac{z-p}{1-\bar pz}\right)\left(\frac{z-q}{1-\bar q z} \right),\text{ where }|p|>1,\;|q|>1.
\end{equation}
The following is shown in \cite{Zakeri}:
\begin{thm}
  \label{thm:zakeri1}
  Any Blaschke product which is topologically conjugate to an element of the family (\ref{eq:blaschke}) by a conjugacy which
  preserves $0$, $\infty$, and $S^1$  has the form $B\circ R$ where $B$ is an element of the family (\ref{eq:blaschke}) and $R(z)=e^{ir}z,\; r\in\RR$.
  \end{thm}

  \begin{thm}\label{thm:zakeri2}
  Let $\rho\notin \QQ$ and $c\in(0,1)$. Then there exists a unique Blaschke product $B(z)$ of type (\ref{eq:blaschke}) such that the following is true:
  \begin{itemize}
  \item $\rho(B|_{S^1})=\rho$;
  \item  $B(z)$ is a bi-cubic critical circle map (one of the critical points is at $e^{2\pi i 0}=1$) and the unmarked combinatorial type of $B$ is given by $(\rho, \{3,3\}, \{c,1-c\})$.
    \end{itemize}
\end{thm}

  We now prove:
  \begin{lem}
    \label{lem:deform}
    Theorem \ref{hybrid} holds under the assumption that 
    $f=B$ is as in Theorem~\ref{thm:zakeri2}.
  \end{lem}
  \begin{proof}
Extend $\mu$ by the dynamics of $B$ to a $B$-invariant, $S^1$-symmetric Beltrami differential supported on 
$$\overline{\cup_{k\geq 0}B^{-k}(K(\cH))}=J(B)$$
(the equality is implied by the standard properties of Julia sets, cf. e.g. \cite{Mil}). Complete it to a $B$-invariant Beltrami diffential $\lambda_1$ in $\hat\CC$ by setting it equal to $\sigma_0$ on the complement of the Julia set of $B$. For $t\in[0,1]$, set
$$\lambda_t=t\lambda_1,$$
so that $\lambda_0=\sigma_0$. By Ahlfors-Bers-Boyarski Theorem, there exists a unique solution $\psi_t:\hat\CC\to\hat\CC$ of the Beltrami equation
$$\psi_t^*\sigma_0=\lambda_t$$
which fixes the points $0$, $1$, and $\infty$ and continuously depends on $t$. By $S^1$-symmetry, the rational map
$$G\equiv \psi_t\circ B\circ\psi_t^{-1}$$
is a Blaschke product. By Theorem~\ref{thm:zakeri1} and Theorem~\ref{thm:zakeri2},
$$G=B.$$
Hence, for every $t$ and $m$, the map $\psi_t$ permutes the discrete set $B^{-m}(1)$. Since $\psi_0=\text{Id}$ and $\psi_t$ continuously depends on $t$, $\psi_t$ fixes the points in each  $B^{-m}(1)$. By the standard properties of Julia sets,
$$\overline{\cup_{m\geq 0}B^{-m}(1)}=J(f).$$
Hence, for all $t\in[0,1]$, 
$$\psi_t|_{J(B)}=\text{Id}.$$
By Bers Sewing Lemma,
$$\bar\partial\psi_t=t\lambda_1=0\text{ a.e.}$$

    \end{proof}
\begin{proof}[Proof of Theorem~\ref{hybrid}]
By Theorem~\ref{thm:zakeri2}, and complex {\it a priori} bounds, there exists a critical circle map $B$ such that the renormalization $\cR^k(B)$ extends to a holomorphic commuting pair $\cG$ which is $K$-quasiconformally conjugate to $\cH$. Denote $\psi$ such a conjugacy. The claim will follow from the above considerations if we can show that $\bar\partial\psi=0$ a.e. on $K(\cH)$. Indeed, assume the contrary. Then 
$(\psi^{-1})^*\sigma_0$ is a non-trivial invariant Beltrami differential on $K(\cG)$, which contradicts Lemma~\ref{lem:deform}.
\end{proof}

\section{Cylinder renormalization of bi-cubic maps}
\label{section: cyl ren}

\subsection{Some functional spaces}
Firstly, let us denote $\pi$ the natural projection $\CC\to\CC/\ZZ$. For an equatorial annulus $U\subset \CC/\ZZ$
let ${\aaa A}_U$ be the space of bounded analytic maps $\phi:U\to \CC/\ZZ$ continuous up to the boundary, such that 
$\phi(\TT)$ is homotopic to $\TT$, equipped with the uniform metric.
We shall turn ${\aaa A}_U$ into a real-symmetric complex Banach manifold as follows. 
Denote $\tl U$ the lift $\pi^{-1}(U)\subset \CC$. The space of functions $\tl\phi:\tl U\to\CC$ which are analytic,
continuous up to the boundary, and $1$-periodic, $\tl\phi(z+1)=\tl\phi(z)$, becomes a Banach space when endowed
with the sup norm. Denote that space $\tl{\aaa A}_U$. For a function $\phi:U\to\CC/\ZZ$ denote
$\check\phi$ an arbitrarily chosen lift $\pi(\check\phi(\pi^{-1}(z)))=\phi$.
Observe that $\phi\in{\aaa A}_U$ if and only if $\tl\phi=\check\phi-\operatorname{Id}\in\tl{\aaa A}_U$.
We use the local homeomorphism between $\tl{\aaa A}_U$ and ${\aaa A}_U$ given by
$$\tl\phi\mapsto \pi\circ(\tl\phi+\operatorname{Id})\circ\pi^{-1}$$ 
to define the atlas on  ${\aaa A}_U$. The coordinate change transformations
are given by $\tl\phi(z)\mapsto \tl\phi(z+n)+m$ for $n,m\in\ZZ$, therefore with this atlas 
${\aaa A}_U$ is a real-symmetric complex Banach manifold.

Let us denote $\mathbf f$ the critical circle map
$$\mathbf f(z)=z-\frac{1}{2\pi}\sin(2\pi z) \mod \ZZ.$$

\noindent
\begin{defn}
  Let $\cU=(U_1,U_2,U_3)$ be an ordered triple of $\RR/\ZZ$-symmetric equatorial annuli. We denote
$\mathbf C_\cU$ the space of triples $\phi_i\in{\aaa A}_{U_i}$, $i=1,2,3$ such that:
  \begin{itemize}
  \item $\mathbf f\circ \phi_1(U_1)\Subset U_2$ and $\mathbf f\circ \phi_2\circ \mathbf f\circ \phi_1(U_1)\Subset U_3$;
    \item $\phi_1(0)=0$;
    \item and each $\phi_i$ is univalent in some neighborhood of the circle.
    \end{itemize}

\end{defn}
Clearly, $\mathbf C_\cU$ is an open subset of ${\aaa A}_{U_1}\times {\aaa A}_{U_2}\times {\aaa A}_{U_3}$, and thus possesses a natural product Banach manifold structure. Note that the composition
\begin{equation}
  \label{cyl-map}
g_{(\phi_i)}\equiv \phi_3\circ\mathbf f\circ \phi_2\circ \mathbf f\circ \phi_1:U_1\to U_3
  \end{equation}
is a bi-cubic circle map in ${\aaa A}_{U_1}$. We note:
\begin{prop}
  \label{projection1}
  The correspondence $\Gamma:\mathbf C_\cU\to {\aaa A}_{U_1}$, given by
  $$(\phi_i)\mapsto g_{(\phi_i)}$$
  is analytic.
  
\end{prop}

\noindent
\begin{defn}
  \label{cylren}
Given an element $\mathbf v\in \mathbf C_\cU$ we say that 
it is {\it cylinder renormalizable} with period $k$ if the following holds.
Denote $f=\Gamma(\mathbf v)$. Then, 
\begin{itemize}
\item there exists $m\in \NN$ such that the rotation number $\rho(f)$ has at least $m+1$ digits in its continued fraction expansion, and
  $k=q_m$;
\item there exist repelling periodic points $p_1$, $p_2$ of $f$ in $U_1$ with periods $k$
and a simple arc $l$ connecting them such that $f^k(l)$ is a simple arc, and  $f^k(l)\cap l=\{p_1,p_2\}$;
\item the iterate $f^k$  is defined and univalent in the domain $C_f$  bounded by $l$ and $f^k(l)$,
the corresponding inverse branch $f^{-k}|_{f^k(C_f)}$ univalently extends to $C_f$;
and the quotient of $\overline{C_f\cup f^k(C_f)}\sm\{p_1,p_2\}$ by the action of $f^k$ is a 
Riemann surface conformally isomorphic to the cylinder $\CC/\ZZ$ 
(we will  call a domain $C_f$ with these properties
a {\it fundamental crescent of $f^k$});
\item for a point $z\in \bar C_f$ with $\{f^j(z)\}_{j\in\NN}\cap \bar C_f\ne \emptyset$,
set $R_{C_f}(z)=f^{n(z)}(z)$ where $n(z)\in\NN$ is the smallest value for which $f^{n(z)}(z)\in\bar C_f$.
We further require that 
there exists a point $c$ in the domain of $R_{C_f}$ such that $f^m(c)=0$ for some $m<n(c)$.
\end{itemize}

Denote $\hat f$ the projection of $R_{C_f}$ to $\CC/\ZZ$ with $c\mapsto 0$. Evidently, it is a bi-cubic circle map.
We will say that $\hat f$ is a {\it cylinder renormalization} of $f$
with period $k$.
\end{defn}

\noindent
The relation of the cylinder renormalization procedure to critical circle maps is easy to see:
\begin{prop}[\cite{Ya3}]
  \label{prop cylren1}
Suppose $f$ is a critical circle map with rotation number $\rho(f)\in \RR\setminus\QQ$.
Assume that it is cylinder renormalizable with period $q_n$. Then the corresponding renormalization
$\hat f$ is also a critical circle map with rotation number $G^n (\rho(f))$. Also, 
$\cR\hat f$ (in the sense of pairs) is analytically conjugate to $\cR^{n+1}f$.
\end{prop}

\noindent
The next property of  cylinder renormalization is the following:
\begin{prop}[\cite{Ya3}]
\label{ren open set}
Let $f=\Gamma(\mathbf v)$ be renormalizable with period $k=q_n$, with a cylinder renormalization $\hat f$ corresponding to the fundamental 
crescent $C_f$, and let $W$ be any equatorial annulus compactly contained in the domain $U_1$ of $f$.
Then there is an open neighborhood $G$ of $\mathbf v$ such that every $\mathbf u\in G$ the map
$g\equiv\Gamma(\mathbf u)$ is 
renormalizable, with a fundamental crescent $C_g\subset U$ which depends continuously on $\mathbf u$ in the
Hausdorff sense. Moreover, there exists a holomorphic motion 
$\chi_{\mathbf u}:\partial C_f\mapsto \partial C_g$ over $G$, such that 
$\chi_g(f(z))=g(\chi_g(z))$.
And finally, the renormalization $\hat g$ is contained in ${\aaa A}_W$.
\end{prop}

To define cylinder renormalization as an operator acting on triples in $\mathbf C_\cU$, let us first note the following simple statement:
\begin{prop}
  \label{uniquedecomp}
  Set $\hat\cU=(\hat U_1,\hat U_2,\hat U_3)$, and suppose, that there exists a triple $\hat {\mathbf v}\in\mathbf C_{\hat \cU}$ such that
  $\Gamma(\hat{\mathbf v})=\hat f$ as above. Then such an element $\hat{\mathbf v}$ is unique.
\end{prop}

\noindent
The next proposition shows that $\mathbf u\mapsto \hat g$ is well-defined:
\begin{prop}
\label{well-defined}
In the notations of of \propref{ren open set}, let $C'_g$ be a different family of 
fundamental crescents, which also depends continuously on $\mathbf u\in G$ in the Hausdorff
sense. Then there exists an open neighborhood $G'\subset G$ of $f$, such that 
the cylinder renormalization of $\mathbf u\in G'$ corresponding to $C'_g$ is also $\hat g$.
\end{prop}

\noindent
A key property of the cylinder renormalization is the
following:

\begin{prop}
\label{analytic dependence}
In the notations of \propref{ren open set}, the dependence $\mathbf u\mapsto \hat g$ is 
an analytic map $G\to {\aaa A}_W$.
\end{prop}
\begin{pf}
By the Theorem of Bers and Royden \cite{BR}, 
the holomorphic motion $\chi_g$ extends to a holomorphic
motion $\chi_g:C_f\cup f^k(C_f)\to C_g\cup g^k(C_g)$ 
over a smaller open neighborhood $G'$ with the same equivariance property.
As shown in \cite{MSS}, the holomorphic motion induces  an analytic family
of quasiconformal maps 
$$\Psi_g:\CC/\ZZ=(C_f\cup f^k(C_f))/f^k\to \CC/\ZZ=(C_g\cup g^k(C_g))/g^k.$$
Applying the theorem of Ahlfors and Bers, we see that the projection
$\pi_g:\overline{C_g}\cup g^k(C_g)\to\CC/\ZZ$ depends analytically on $g$. 
Let  $D\subset C_f\cup f^{-k}(C_f)$ and $n\in \NN$ be such that 
$W\Subset \pi_f(D)$, and $\hat f\in{\aaa C}_W$ is the projection of the iterate 
$f^n|_D$. The iterate $g^n|_D$ projects to 
$\pi_g\circ g^n=\hat g\in{\aaa C}_W$ for all $g$ sufficiently close to $f$, and the claim follows.
\end{pf}

\noindent
The following is an immediate corollary:

\begin{prop}
  \label{analytic dependence2}
  In the above notation,
 the correspondence $\mathbf v\mapsto\hat{\mathbf v}$ is a locally analytic operator from a neighborhood of $\mathbf v\in\mathbf C_{\cU}$ to
    $\mathbf C_{\hat \cU}$.
   
\end{prop}

\begin{prop}
  \label{conj-maps}
  Let $f$ and $g$ be two bi-critical circle maps which are analytically conjugate
  $$\phi\circ f=g\circ \phi$$
  in a neighborhood $U$ of the circle. Suppose that $f$ is cylinder renormalizable with period $k=q_n$, and let $C_f$ be a corresponding fundamental crescent of $f$. Suppose $C_f\cup f(C_f)\Subset U$. Then, $g$ is also cylinder renormalizable with the same period and with a fundamental crescent  $C_g=\phi(C_f)$; and denoting $\hat f$ and $\hat g$ the corresponding cylinder renormalizations of $f$ and $g$ respectively, we have
  $$\hat f\equiv\hat g.$$
\end{prop}
\begin{proof}
  The domain $C_g$ is a fundamental crescent by definition. Denote $C_f'$, $C_g'$ the closures of the two fundamental crescents minus the endpoints.
  The analytic conjugacy $\phi$ projects to an analytic isomorphism
  $$\hat\phi:C'_f\cup f^k(C'_f)/f^k\simeq\CC/\ZZ\longrightarrow C'_g\cup g^k(C'_g)/g^k\simeq\CC/\ZZ,$$
  which conjugates $\hat f$ to $\hat g$. By Liouville's Theorem, $\hat\phi$ is a translation; since it fixes the origin, it is the identity.
\end{proof}

\subsection{Cylinder renormalization operator}

The cylinder renormalization of a bi-cubic commuting pair is defined in the same way as that of an analytic bi-cubic map.:
\begin{defn}
Let $\zeta=(\eta,\xi)$ be a bi-cubic holomorphic commuting pair which is at least $n$ times renormalizable for some $n\in\NN$.
Denote $D_\eta$ the domain of the pair $\zeta$.
Denote $$\zeta_n=(\eta_n,\xi_n)\equiv p\cR^n(\zeta):D_{\eta_n}\to\CC.$$
We say that $\zeta$ is cylinder renormalizable with period $k=q^n$ if:
 \begin{itemize}
\item there exist two compex conjugate fixed points $p^+_n$, $p^-_n$ of the map $\eta_n$ with $p^\pm_n\in\pm\HH$,
and an $\RR$-symmetric simple arc $l$ connecting these points such that $l\subset {D_{\eta_n}}$ and 
$\eta_n(l)\cap l=\{p^+_n,p^-_n\}$;
\item denoting $C$ the $\RR$-symmetric domain bounded by $l$ and $\eta(l)$ we have $C\subset D_{\eta_n}$,
and the quotient of $\overline{C\cup\eta(C)}\sm\{p^+_n,p^-_n\}$ 
by the action of $\zeta_n$ is a Riemann surface homeomorphic
to the cylinder $\CC/\ZZ$ (that is, $C$ is a fundamental crescent of $\eta$).
\end{itemize}
Clearly, the action of $\zeta$ induces a bi-critical circle map $f$ with rotation number $\rho(\cR^n\zeta)$ on the quotient $$(\overline{C_\zeta\cup\eta(C_\zeta)}\sm\{p^+_n,p^-_n\})/\zeta_n\simeq \CC/\ZZ;$$
we call $f$ the cylinder renormalization of $\zeta$.

\end{defn}

\begin{thm}
  \label{exist crescent}
  
  Fix a periodic marked combinatorial type $\cC$, and let $f\in{\aaa A}_U$ be a bi-cubic map of this type. Then, there exists $N=N(\cC,U)$
  such that the following holds. For every $n\geq N$, setting $k=q_n$, $f$ is cylinder renormalizable with period $k$. Moreover, there exists a fundamental crescent $C_f\Subset U$ (which moves analytically with $f$) with period $k$ such that $f^k(C_f)\subset U$, and, denoting $\hat f$ the corresponding cylinder renormalization of $f$, we have
$$\hat f\in{\aaa A}_U.$$

Similarly, let $\zeta$ be any bi-cubic holomorphic commuting pair  of type $\cC$. Then, there exists $k=k(\zeta)=q_n$ such that $\zeta$ is cylinder renormalizable with period $k$, and the cylinder renormalization is a bi-cubic map of type $\cC$.

\end{thm}
The proof is a direct consequence of complex {\it a priori} bounds, and follows {\it mutatis mutandis} the proof of Lemma 4.10 of \cite{GorYa}, so we will not reproduce it here.

\begin{prop}
  \label{renorm-operator}
  Let $\cC$ be a periodic marked type with period $p$ and let $\zeta_*$ be a periodic point of $\cR$ with this period. Then, there exist $n\in\NN$, an equatorial annulus $U_*$, and a triple $\cU=(U_*,U_2,U_3)$ such that the following holds:
  \begin{enumerate}
  \item there exists $m$ such that $\zeta_*$ is cylinder renormalizable with period $m$ and the renormalization $f_*\in {\aaa A}_{U_*}$ has combinatorial type $\cC$;
  \item the bi-cubic map $f_*$ is cylinder renormalizable with period $k$; its renormalization is equal to $f_*$ and is defined in $U_1\Supset U_*$;
    \item the map $f_*$ has a decomposition ${\mathbf v}_*\in \mathbf C_\cU$.

    \end{enumerate}
  \end{prop}
The last property implies that there exists a neighborhood $V( {\mathbf v}_*)\subset \mathbf C_\cU$ such that cylinder renormalization with period $k$ is an analytic operator to ${\mathbf C}_\cU$. We call this operator {\it cylinder renormalization operator} corresponding to the type $\cC$ and denote it
$$\cren:V( {\mathbf v}_*)\longrightarrow \mathbf C_\cU.$$

\section{Constructing the stable manifold of $\mathbf v_*$}
We will now specialize to real-symmetric maps, and consider the real slice $\curr$ of the space ${\aaa C}_\cU$ consisting of maps with real symmetry, with the induced real Banach manifold structure (cf. \cite{Ya3,GorYa}).
We proceed with the above notation. Let $\rho_*$ be the rotation number of the combinatorial class $\cC$
and define $$D_*=\{ \mathbf v\in\curr,\text{ such that }\rho(\Gamma(\mathbf v))=\rho_*\},$$
and
$$S_*=\{\mathbf v\in\curr,\text{ such that }\cC(\Gamma(\mathbf v))=\cC\}\subset D_*.$$
As we have shown,
\begin{prop}
\label{local stable}
There exists a neighborhood $Y$ of $\fxpt$ in $\curr$ such that for every 
${\mathbf v}\in Y\cap S_*$
$\cren^j({\mathbf v})$ is defined for all $j$, and $$\cren^j({\mathbf v})\longrightarrow\fxpt$$ uniformly in $Y$.
\end{prop}

\noindent
Below we shall demonstrate that the local stable set of $\fxpt$ is a submanifold of $\mfld$:
\begin{thm}
\label{stable manifold}
There is an open neighborhood $W\subset \curr$ of $\fxpt$ such that $S_*\cap W$
is a smooth submanifold of $\mfld$ of codimension $2$. Moreover, for $\bfv\in S_*$ denote $f=\Gamma(\bfv)$, and let $\phi_f$ be the unique
smooth conjugacy between $f$ and $f_*$ which fixes $0$. Then the dependence $\bfv\mapsto \phi_f$ is smooth.
\end{thm}

\noindent
Our first step will be to show that
\begin{thm}\label{stable1}
  The set $D_*$ is a local submanifold of $\curr$ of codimension $1$.
  \end{thm}
As before denote $p_k/q_k$ the reduced form of the $k$-th
continued fraction convergent of $\rho_*$. Furthermore, 
define $D_k$ as the set of ${\mathbf v}\in\cu$ for which $\rho(g)=p_k/q_k$ and 
$0$ is a periodic point of $g$ with period $q_k$, where $g=\Gamma({\mathbf v}).$
As follows from the Implicit Function Theorem, this is a codimension $1$ submanifold of $\cu$.
Let us define a cone field $\cE$ in the tangent bundle of $\cu$ given by the condition:
$$\cE=\{\nu\in T\cu|\; \inf_{x\in\RR}D\Gamma \nu (x)>0\}.$$
\begin{lem}
\label{cone not in tk}
Let ${\mathbf v}\in D_k$, and denote $T_{{\mathbf v}}D_k$ the tangent space to $D_k$ at this point.
Then $T_{{\mathbf v}}D_k\cap \cE=\emptyset$.
\end{lem}
\begin{pf}
Let $\nu\in \cE$ and suppose $\{ {\mathbf v}_t\}\subset \cu$ is a one-parameter family
such that  $${\mathbf v}_t={\mathbf v}+t\nu+o(t),$$
and set $f=\Gamma({\mathbf v})$ and $f_t=\Gamma({\mathbf v}_t)$, so that
$$\pi^{-1}\circ f_t\circ\pi=\pi^{-1}\circ f\circ\pi+tD\Gamma(\nu)+o(t).$$
Then for sufficiently small values of $t$, 
$\pi^{-1}\circ f_t\circ \pi>\pi^{-1}\circ f\circ\pi$. Hence $f_t^{q_k}(0)\neq f^{q_k}(0)$
and thus $f_t\notin D_k$.
\end{pf}


\noindent
Elementary considerations of the Intermediate Value Theorem imply that 
for every $k$ there exists a value of $\theta\in(0,1)$ such that 
the map $f_\theta=R_{\theta}\circ \Gamma(\fxpt)\in D_k$. 
Moreover,
if we denote $\theta_k$ the angle  with the smallest absolute value satisfying this 
property, then $\theta_k\to 0$.
Set $${\mathbf v}_k=(\phi_1,\phi_2,R_{\theta_k}\circ \phi_3).$$
For $k$ large enough, ${\mathbf v}_k\in\cu\cap D_k$.
Let $T_k=T_{{\mathbf v}_k}D_k$.
Fix $v\in\cE$. By  \lemref{cone not in tk} and the Hahn-Banach Theorem there exists $\eps>0$
such that for every $k$ there exists a linear functional $h_k\in (T\cu)^*$
with $||h_k||=1$,
such that $\operatorname{Ker}h_k=T_k$ and $h_k(v)>\eps$.
By the Alaoglu Theorem, we may select a subsequence
$h_{n_k}$ weakly converging to $h\in (T\cu)^*$. Necessarily 
$v\notin \operatorname{Ker}h$, so
$h\not\equiv 0$. Set $T=\operatorname{Ker}h$.

\smallskip
\noindent
{\it Proof of \thmref{stable1}.} 
By the above, we may select a splitting $T\cu=T\oplus v\cdot\RR$.
Denote $p:T\cu\to T$ the corresponding projection, and let
$\psi:\cu\to T\cu$ be a local chart at $\fxpt$.
\lemref{cone not in tk} together with the Mean Value Theorem imply that
$p\circ \psi:D_k\to T$ is an isomorphism onto the image,
and there exists an open neighborhood $\cU$ of the origin in $T$, such that 
$p\circ \psi(D_k)\supset \cU$. Since the graphs $D_{k}$ are analytic, we may
select a $C^1$-converging subsequence $D_{k_j}$ whose limit is a smooth graph $G$ over $\cU$.
 Necessarily, for every $g\in G$, $\rho(g)=\rho_*$. As we have seen above,
every point $g\in D_*$ in a sufficiently small neighborhood of $\fxpt$ is in
$G$, and thus $G$ is an open neighborhood in $D_*$.
$\Box$

We now proceed to describing the manifold structure of $S_*\subset D_*$. Let $c_f\neq 0$ denote the non-zero critical point of $f$. 
Define a cone field $\cE_1$ in $T\cu$ consisting of tangent vectors $\nu$ in $T_\bfv\cu$ such that setting $\upsilon=D\Gamma\nu$, we have
$$\upsilon'(c_f)>0.$$
Furthermore, for $k>n+2$, let $D_{k,n}\subset D_k$  consist of maps $f$ such that
$f$ has the same combinatorics as $\cE$ up to the level $n$, and the $n$-th renormalization $\cR^n(f)$ is uni-critical (that is, two critical orbits collide). Elementary considerations again imply that for every $f\in S_*$ there exists a sequence $\{f_{k,n(k)}\}$ converging to it.

Since the $n$-th renormalization of $f\in D_{k,n}$ is a uni-critical circle mapping, we can apply the results of \cite{Ya4} to conclude:
\begin{thm}
  \label{conj4}
  All maps $f\in\Gamma(D_{k,n})$ are smoothly conjugate. Moreover, let $\hat f$ be an arbitrary base point in $D_{k,n}$. Then, for $f\in D_{k,n}$  the conjugacy $\phi_f$ which realizes
  $$\phi_f\circ f\circ\phi_f^{-1}=\hat f$$
 and sends $0$ to $0$ can be chosen so that the dependence $f\mapsto \phi_f$ is analytic in $S_n$.
  \end{thm}

We note:
\begin{prop}\label{propconj2}
  The set $D_{k,n}$ is a local submanifold of $D_k$ of codimension $1$, and the cone field $\cE_1$ lies outside of its tangent bundle.
  \end{prop}
\begin{proof}
  The first statement is a straightforward consequence of Implicit Function Theorem considerations. By Theorem~\ref{conj4}, a vector field $\upsilon\in D\Gamma T_\bfv D_{k,n}$
  satisfies the cohomological equation
  $$\upsilon=\omega\circ f-f'\omega$$
  where $f=\Gamma(\bfv)$ and $\omega$ is a smooth vector field on the circle. In particular, $\upsilon'(c_f)=0$.
\end{proof}

We also note:
\begin{prop}
  \label{cone3}
  The cone field $\cE_1$ intersects $T_\bfv D_*$ at every point $\bfv$.
\end{prop}
\begin{proof}
Evidently, $\cE_1$ has non-zero intersection with $\cE$ and $-\cE$ at every point $\bfv$. The claim follows.
  \end{proof}

Applying the proof of Theorem~\ref{stable1} {\it mutatis mutandis} (replacing $D_k\to D_*$ with $D_{k,n(k)}\to S_*$), we obtain the statement of Theorem~\ref{stable manifold}.

\section{Hyperbolicity of $\bfv_*$}
\label{sec-hyperb}
Let us denote
$$\cL\equiv D\cren|_{\bfv_*}.$$
We have (see \cite{GorYa}):
\begin{prop}
  \label{prop-comp-lin}
The linear operator $\cL$ is compact.
  \end{prop}
In this section we establish:

\begin{thm}
  \label{thm-unstable}
The linear operator $\cL$ has two unstable eigendirections.
  \end{thm}
Together with Theorem~\ref{stable manifold} this has the following corollary:
\begin{thm}
  \label{cor-hyperb}
  The renormalization fixed point $\bfv_*$ is hyperbolic. It has two unstable eigendirections. Its local stable manifold $W^s_{\text{loc}}(\bfv_*)$ coincides with $S_*$ and is an analytic submanifold of codimension $2$.
  \end{thm}

The following was shown in \cite{GorYa} (see also \cite{Ya3}):
\begin{prop}
\label{dir1}
  There exist constants $a_1>0$ and $\lambda_1>1$ such that for every $\nu\in \cE\cap T_{\bfv}\curr$  and $n\in\NN$
  we have
  $$||\cL^n\nu||>a_1\lambda_1^n\inf_{x\in\RR} D\Gamma\nu(x).$$
\end{prop}

We also have:
\begin{prop}
  \label{dir2}
  There exist constants $a_2>0$ and $\lambda_2>1$ such that the following holds.
  Let $\nu\in T_{\bfv_*}D_*\cap \cE_1$. Set $\upsilon=D\Gamma\nu$ and $\upsilon^\pm_n=D\Gamma\cL^n(\pm\nu)$ for $n\in\NN$. Then, 
  $$\max(|(\upsilon^+_n)'(c_{\cren^n f_*})|,|(\upsilon^-_n)'(c_{\cren^n f_*})|)>a_2\lambda_2^n|\upsilon'(c_{f_*})|.$$
  \end{prop}
\begin{proof}
  The composition $f\mapsto f^n$ transforms the vector field $\upsilon$ into $\hat\upsilon_n$.
  An easy induction shows that
  $$\hat\upsilon_n'(c_{f_*})=\frac{df_*^{n-1}}{dx}(f_*(c_{f_*}))\upsilon'(c_{f_*}).$$
  By real {\it a priori} bounds, the derivative $\frac{df_*^{n-1}}{dx}(f_*(c_{f_*}))$ is uniformly bounded below by $C_0>0$ for $n=q_m$.
  Let $\Phi_n$ stand for the uniformizing coordinate of the fundamental crescent of the $n$-th cylinder renormalization of $f_*$. Then, either for  $u=\upsilon^+_n$, or for $u=\upsilon^-_n$ (depending on the sign of the remainder terms), we can bound below the value of
  $|u'(c_{\cren^n f_*})|$ by
  $$\min_{x\in\RR}\Phi_n'(x)C_0|\upsilon'(c_{f_*})|.$$
  By Koebe Distortion Theorem combined with real {\it a priori} bounds, $\min_{x\in\RR}\Phi_n'(x)$ grows exponentially with $n$, and the claim follows. 

  \end{proof}

\bibliographystyle{amsalpha}
\bibliography{biblio}
\end{document}